\documentclass[12pt]{amsart}
\usepackage{amsmath,amscd,amssymb,euscript}
\usepackage{latexsym}
\usepackage{graphicx}
\usepackage{tikz}
\usetikzlibrary{backgrounds}
\usetikzlibrary{decorations.fractals}
\usetikzlibrary{calc,intersections,through,backgrounds}
\usepackage{mathrsfs}
\usepackage{epstopdf}
\usepackage{rotating}
\usepackage{lscape}
\input xy
\xyoption{all}
\newcommand{\del}{\partial}
\newcommand{\bC}{{\mathbb C}}

\newcommand{\bH}{{\mathbb H}}
\newcommand{\bK}{{\mathbb K}}
\newcommand{\bL}{{\mathbb L}}

\newcommand{\bP}{{\mathbb P}}
\newcommand{\bQ}{{\mathbb Q}}

\newcommand{\bZ}{{\mathbb Z}}
\newcommand{\cA}{{\mathcal A}}

\newcommand{\cC}{{\mathcal C}}

\newcommand{\cF}{{\mathcal F}}

\newcommand{\cO}{{\mathcal O}}
\newcommand{\cP}{{\mathcal P}}

\newcommand{\cR}{{\mathcal R}}

\newcommand{\cZ}{{\mathcal Z}}

\newcommand{\tC}{\widetilde{C}}

\newcommand{\tX}{\widetilde{X}}

\newcommand{\ra}{\rightarrow}
\newcommand{\lra}{\longrightarrow}

\newcommand{\inj}{\hookrightarrow}
\newcommand{\surj}{-\hspace{-4pt}\rightarrow\hspace{-19pt}{\rightarrow}\hspace{3pt}}

\newcommand{\T}{\Theta}

\newcommand{\im}{\operatorname{im}}
\newcommand{\coker}{\operatorname{coker}}

\newcommand{\Pic}{\operatorname{Pic}}

\theoremstyle{definition}

\newtheorem{proposition}{Proposition}[section]
\newtheorem{lemma}[proposition]{Lemma}

\newtheorem{theorem}[proposition]{Theorem}
\newtheorem{theoremi}{Theorem}
\newtheorem{conjecture}[theoremi]{Conjecture}
\newtheorem{definition}[proposition]{Definition}
\newtheorem{corollary}[proposition]{Corollary}

\numberwithin{equation}{section}

\begin{document}

\title{The primitive cohomology of theta divisors}

\author{E. Izadi}

\address{Department of Mathematics, University of California San Diego, 9500 Gilman Drive \# 0112, La Jolla, CA 92093-0112, USA}

\email{eizadi@math.ucsd.edu}

\author{J. Wang}

\address{Department of Mathematics, Boyd
Graduate Studies Research Center, University of Georgia, Athens, GA
30602-7403, USA}
\email{jiewang@math.uga.edu}

\dedicatory{Dedicated to Herb Clemens}

\thanks{The first author was partially supported by the National
Science Foundation. Any opinions, findings
and conclusions or recommendations expressed in this material are those
of the authors and do not necessarily reflect the views of the
NSF}


\begin{abstract}

  The primitive cohomology of the theta divisor of a principally
  polarized abelian variety of dimension $g$ is a Hodge structure of
  level $g-3$. The Hodge conjecture predicts that it is contained in
  the image, under the Abel-Jacobi map, of the cohomology of a family
  of curves in the theta divisor. We survey some of the results known about this primitive cohomology, prove a few general facts and mention some interesting open problems.

\end{abstract}

\maketitle

\tableofcontents

\section*{Introduction}

Let $A$ be an abelian variety of dimension $g$ and let $\T\subset A$ be a theta divisor. In other words, $\T$ is an ample divisor such that $h^0 (A, \T) =1$. We call the pair $(A, \T)$ a principally polarized abelian variety or ppav, with $\T$ uniquely determined up to translation. In this paper we assume $\T$ is smooth.

The primitive cohomology $\bK$ of $\T$ can be defined as the kernel of Gysin push-forward $H^{g-1} (\T, \bZ) \ra H^{g+1} (A, \bZ)$ (see Section 1 below). We shall see that $\bK$ inherits an integral Hodge structure of level $g-3$ from the cohomology of $\T$. 

Recall that the level of an integral or rational Hodge structure $\bH$ is defined to be
\[
l(\bH) := \max \{ |p-q| : \bH^{p,q} \neq 0 \}
\]
where $\bH^{p,q}$ is the $(p,q)$ component of the Hodge decomposition of $\bH_{\bC}$. Alternatively, we define the coniveau of $\bH$ to be 
$$\gamma(\bH):=\min\{q:\bH^{p,q} \neq 0\}.$$
Thus 
\[
l(\bH)+2\gamma(\bH)=\text{weight of }\bH.
\]

We will always designate a Hodge structure by its lattice or rational
vector space $\bH$, the splitting $\bH_{\bC} := \bH\otimes\bC =\oplus_{ p+q =m} \bH^{p,q}$
being implicit.

The general Hodge conjecture says that $\bK_{\bQ}$ is contained in the image of the cohomology of a proper algebraic subset of $\T$. More precisely, let $X$ be a smooth projective algebraic variety and $m$ and $p$ two positive integers with $2p \leq m$. Grothendieck's version of the general Hodge conjecture \cite{grothendieck69} can be stated as

\begin{conjecture}
{\bf GHC(X,m,p):} For every rational sub-Hodge structure $V$ of $H^m
(X,\bQ )$ with level $\leq m- 2p$, there exists a closed algebraic subset $Z$ of
$X$ of pure codimension $p$ such that $V\subset Ker\{ H^m (X,\bQ )\ra H^m (
X\setminus Z,\bQ )\}$.
\end{conjecture}

In the case of $\bK_{\bQ}$, we have $X=\T$, $m=g-1$ and $p=1$. We are therefore looking for a divisor in $\T$.

The general Hodge conjecture for $\bK_{\bQ}$ can be answered positively for $g\leq 5$. Here we survey these results and the tools used to obtain them. We also say a few words about higher dimensional cases and mention other interesting problems related to the primitive cohomology group $\bK$.

There are relatively few examples of lower level sub-Hodge structures of the cohomology of algebraic varieties that are not already contained in the images of the cohomology groups of subvarieties for trivial reasons. Some of the most interesting such examples are provided by abelian varieties, such as abelian varieties of Weil type (see, e.g., \cite{I11}).

For a smooth hypersuface $Y$ of degree $d$ in $\bP^g$, the primitive cohomology $H^{g-1}(Y,\bQ)_0$ is a sub-Hodge structure of coniveau at least $1$ if and only if $Y$ is Fano, since $H^{g-1,0}(Y)\cong H^0(Y,K_Y)$. Thus, if $Y$ is of general type or Calabi-Yau, i.e. $d\ge g+1$, $H^0(Y,\bQ)_0$ is of coniveau $0$.  On the other hand, the general Hodge conjecture is true for Fano hypersurfaces with coniveau $1$ primitive cohomology group (see, e.g.,\cite{voisin11}).

\section{General considerations}

There is a strong relation between the cohomology of $A$ and $\T$. For instance, one has the Lefschetz hyperplane theorem:

\begin{theorem}\label{j*j!}
Let $j : \T \inj A$ be the inclusion. Then
\[
\begin{array}{ll}
j_{*}: H_k(\Theta, \bZ) \lra H_{k}(A, \bZ)&j_{!}:H^{k}(\Theta, \bZ) \lra H^{k+2}(A, \bZ) \\
j^{*}:H^{k}(A, \bZ) \lra H^{k}(\Theta, \bZ) & j^{!}: H_{k+2}(A, \bZ) \lra H_{k}(\Theta, \bZ)\\
{\rm are\; isomorphisms\; for\;} k < g-1, \quad & {\rm are\; isomorphisms\; for\; }k
> g-1.
\end{array}
\]
Also $j_*$ and $j_!$ are surjective for $k=g-1$, $j^*$ and $j^!$ are
injective for $k=g-1$.
The maps $j^!$ and $j_!$ are defined to be
$$ j^{!}:= P_{\Theta}.j^{*}.P_{A}^{-1}, \; \;j_{!}:=
P_{A}^{-1}.j_{*}.P_{\Theta}, \; \; $$
where $P_{\T}:H^{k}(\T, \bZ) \lra H_{2g-2-k}(\T, \bZ)$,
$P_{A}:H^{k}(A, \bZ) \lra H_{2g-k}(A, \bZ)$ are the Poincar\'e duality maps. We also have
\[
\cup \theta=j_{!}.j^{*},
\]
where $\cup \theta$ is the cup product with the fundamental class of $\T$.
\end{theorem}

\begin{proof}
See, e.g., \cite{andreottifrankel59}.
\end{proof}

It is well-known (see, e.g., \cite{I4} Proposition 1.1) that the integral cohomology and homology groups of $\T$ and $A$ are torsion-free.

The cohomology of $\T$ is therefore determined by that of $A$ except in degree $g-1$. Following \cite{I4} and \cite{I17}, define
\[
\bK := \ker \{ j_{!}:H^{g-1}(\Theta, \bZ) \surj \; H^{g+1}(A, \bZ) \}
\]
so that its dual lattice is (see \cite{I4} Proposition 1.3)
\[
\bK^* = \coker \{ j^{*}:H^{g-1}(A, \bZ) \inj H^{g-1}(\Theta, \bZ) \}.
\]

\begin{lemma}
The rank of $\bK$ is
\[
rank (\bK) = g! - \frac{1}{g+1}{2g \choose g}.
\]
\end{lemma}
\begin{proof}
It follows easily from
Theorem \ref{j*j!} that the rank of $\bK$ is
\[
rank (\bK)=h^{g+1}(A)-h^{g}(A)+(-1)^{g-1}(\chi_{top}(\T)-\chi_{top}(A)).
\]
Using the exact sequence 
\[0\lra T_{\T}\lra T_A|_{\T}\lra \cO_{\T}(\T)\lra 0,
\]
we see that the total Chern class $c(T_{\T})$ of the tangent bundle of $\T$ satisfies the identity
\[
c(T_{\T})(1+\theta|_\T)=1.
\]
Therefore
\[
\chi_{top}(\T) =\deg c_{g-1}(T_{\T})=  (-1)^{g-1}g!.
\]
As $h^{g+1}(A) = {2g \choose g+1}$ and $h^g(A) = {2g \choose g}$, we find
\[
rank (\bK) = g! - \frac{1}{g+1}{2g \choose g}
\]
as claimed.
\end{proof}

The integral lattices $\bK$ and $\bK^*$ inherit Hodge structures from $H^{g-1} (\T, \bZ)$.
One can use Griffiths' residue calculus \cite{griffiths69} to compute all the Hodge summands of $\bK$ as follows. 

Put $U :=A\setminus\T$ and let $i:U\ra A$ be the natural inclusion. Also, for an integer $k \in \{ 0, \ldots , g \}$, let
\[
H^{g-k} (A, \bQ)_{prim} := \ker \{ \cup \theta^{k+1} : H^{g-k} (A, \bQ ) \lra H^{g+k+2} (A, \bQ) \}
\]
be the $(g-k)$-th primitive cohomology group of $A$. The Gysin exact sequence (see \cite[p. 159]{voisin03})
\[
\lra H^{g-2}(\T,\bQ)\stackrel{j_!}\lra H^g(A,\bQ)\stackrel{i^*}\lra H^g(U,\bQ)\stackrel{Res}\lra H^{g-1}(\T,\bQ)\stackrel{j_!}\lra H^{g+1}(A,\bQ)\lra
\]
induces a short exact sequence of mixed Hodge structures
\[
0\lra H^g(A,\bQ)_{prim}\lra H^g(U,\bQ)\stackrel{Res}\lra\bK_\bQ\lra0.
\]
Thus for $0\le p\le g$, the induced sequence on the Hodge filtration 
\begin{eqnarray}\label{eqnHodge}
0\lra H^{g-p,p}(A)_{prim}\lra \frac{F^{g-p}H^g(U)}{F^{g-p+1}H^g(U)}\lra \bK^{g-p-1,p}\lra0
\end{eqnarray}
is exact.

Griffiths' residue calculus implies that there is an exact sequence 
\begin{eqnarray}\label{eqnGriff}
H^0(\Omega^{g-1}_A(p\T))\stackrel{d}\lra H^0(\omega_A((p+1) \T))\stackrel{\alpha_{p}}\lra F^{g-p}H^g(U)\lra 0
\end{eqnarray}
where the leftmost map is the exterior derivative and the middle map sends a rational $g$-form on $A$ with a pole of order $\leq p+1$ on $\T$ to its De-Rham class in $U$ (c.f. \cite[pp. 160-162]{voisin03}).

Denote 
\[
\overline{\alpha}_{p}: H^0(\T,\cO_\T((p+1)\T))\cong \frac{H^0(\omega_A((p+1)\T))}{H^0(\omega_A(p\T))}\lra \frac{F^{g-p}H^g(U)}{F^{g-p+1}H^g(U)}
\]
the induced map.

\begin{lemma}
The Hodge structure on $\bK$ satisfies $\bK^{g-1,0}=0$ and $\dim_\bC\bK^{g-2,1}=2^g-1-\frac{g(g+1)}{2}$. Thus $\bK$ and $\bK^*$ have level $g-3$ ($g\ge3$).
For $p\ge2$, we have an exact sequence
\[
H^0(\T,\cO_\T(\T))\otimes H^0(\T,\cO_\T(p\T))\lra H^0(\T,\cO_\T((p+1)\T))\stackrel{\overline{\alpha}_{p}}\lra \frac{F^{g-p}H^g(U)}{F^{g-p+1}H^g(U)}\lra0,
\]
i.e. $\frac{F^{g-p}H^g(U)}{F^{g-p+1}H^g(U)}$ is isomorphic to the Koszul cohomology group $K_{0,p+1}(\T,\cO_\T(\T))\cong K_{0,p+1}(\T,K_\T)$ (see, e.g., \cite{green841}). 
\end{lemma}
\begin{proof}
 When $p=0,1$, the image of the exterior derivative in (\ref{eqnGriff}) is zero. We conclude from (\ref{eqnGriff}) that $H^0(\omega_A(\T))\cong F^gH^g(U)\cong \bC$ and 
\[
H^0(\T,\cO_\T(2\T)){\cong}\frac{F^{g-1}H^g(U)}{F^{g}H^g(U)}.
\]
Therefore, by (\ref{eqnHodge}), $\bK^{g-1,0}=0$. Since $h^0((\T,\cO_\T(2\T))=2^g-1$ by Riemann-Roch and $h^{g-p,p}(A)_{prim}=\frac{g(g+1)}{2}$, we obtain from (\ref{eqnHodge}) that $\dim_\bC\bK^{g-2,1}=2^g-1-\frac{g(g+1)}{2}$.

Let $Q_{p+1}\in H^0(\cO_A((p+1)\T))$ be such that $\frac{Q_{p+1}}{\theta^{p+1}}dz_1\wedge...\wedge dz_g\in F^{g-p+1}H^g(U)$ where $\{dz_1,...,dz_g\}$ form a basis of $H^0(A,\Omega_A^1)$.
Thus, by (\ref{eqnGriff}), there exists $Q_p\in H^0(\cO_A(p\T))$ such that the rational form
\[
(\frac{Q_{p+1}-\theta Q_p}{\theta^{p+1}})dz_1\wedge...\wedge dz_g=d\gamma
\]
for some $\gamma=\sum_{i=1}^g(\frac{s_i}{\theta^p})dz_1\wedge...\wedge \hat{dz_i}\wedge...\wedge dz_g$, with $s_i\in H^0(\cO_A(p\T))$.

We directly compute
\[
d\gamma=\sum_{i=1}^g(-1)^i \left(\frac{\frac{\partial s_i}{\partial z_i}\theta+(-p)s_i\frac{\partial \theta}{\partial z_i}}{\theta^{p+1}} \right)dz_1\wedge...\wedge dz_g.
\]
Comparing the two sides, we see that 
\[
Q_{p+1}-\theta Q_p=\sum_{i=1}^g(-1)^i \left(\frac{\partial s_i}{\partial z_i}\theta+(-p)s_i\frac{\partial \theta}{\partial z_i} \right).
\]
Restricting the above equality to $\T$, we obtain 
\[
Q_{p+1}|_\T=(-p)\sum_{i=1}^g(-1)^i \left(s_i\frac{\partial \theta}{\partial z_i} \right)|_\T  \quad \in H^0(\T,\cO((p+1)\T) ).
\]
Since $\{\frac{\partial\theta}{\partial z_i}:i=1,...,g\}$ form a basis of $H^0(\T,\cO_\T(\T))$ (see, e.g, \cite[p. 92]{green84}), we conclude our proof.
\end{proof}

For $g\leq 2$, $\bK =0$. For $g=3$, the lattice $\bK$ has rank $1$ and level $0$, i.e., it is generated by a Hodge class of degree $2$. By the Lefschetz $(1,1)$-theorem, this is a rational linear combination of classes of algebraic cycles. In fact, in this case, one can write an explicit cycle generating $\bK$ as follows. The abelian variety $(A, \T) = (JC, \T_C) $ is the Jacobian of a curve of genus $3$. The theta divisor is isomorphic to the second symmetric power $C^{(2)}$ of $C$ and $\bK$ is generated by the class $\theta - 2\eta$ where $\eta$ is the cohomology class of the image of $C$ in $C^{(2)}$ via addition of a point $p$ of $C$:
\[
\begin{array}{ccc}
C & \inj & C^{(2)} \\
t & \mapsto & t+p.
\end{array}
\]
For higher values of $g$, the following equivalent formulation of the Hodge conjecture has been useful (see e.g. \cite{I11}).

\begin{conjecture}\label{hodge3}
There exists a nonsingular projective family of
curves in $\T$
\[
\begin{array}{ccc}
\cZ & \stackrel{q}{\lra} & \T \\
\downarrow^r & & \\
S & &
\end{array}
\]
whose base is a (possibly reducible) nonsingular projective variety $S$ of dimension $g-3$
such that the image of $H^{ g-3 } (S ,\bQ)$ by the Abel-Jacobi map
$q_! r^*$ of the family contains $\bK_{\bQ}$.
\end{conjecture}

For $g=4, 5$, a positive answer was given to the above conjecture by using the ``largest'' sub-Hodge structure $\bH$ of coniveau $1$ of $H^{g-1}(\T, \bZ)$ defined as follows (see \cite{I4} and \cite{I17}).

First consider the image of $H^{g-3} (\T, \bZ) =j^* H^{g-3} (A, \bZ)$ under cup product with the cohomolgy class $\theta$ of $\T$. This is also a sub-Hodge structure of level $g-3$ and satisfies the Hodge conjecture since it is contained in the image, for instance, of the cohomology  of an intersection of a translate of $\T$ with $\T$.

Put
\[
P^{g-1} := H^{g-1}(A, \bZ)_{prim} = \ker \{ \cup \theta^2 : H^{g-1} (A, \bZ) \lra H^{g+3} (A ,\bZ) \}.
\]
Choosing a symplectic basis $\{\alpha_1, \ldots , \alpha_g, \beta_1 , \ldots , \beta_g \}$ of $H^1 (A, \bZ)$, it is immediately seen that the wedge products $\gamma_{i_1} \wedge \ldots \wedge \gamma_{i_{g-1}}$ form a $\bZ$-basis of $P^{g-1}$ where $\gamma_i = \alpha_i$ or $\beta_i$ and $i_1 < \ldots < i_{g-1}$. It follows that the dual of the embedding
\[
P^{g-1} \inj H^{g-1} (A , \bZ)
\]
is a surjection
\[
H^{g+1} (A, \bZ ) \surj \: (P^{g-1})^*
\]
after identifying $H^{g+1} (A, \bZ )$ with the dual of $H^{g-1} (A , \bZ)$ using the intersection pairing.

The Hodge structure $\bH$ can then be defined as the kernel of the composition
\[
\bH := \ker \{ H^{g-1} (\T , \bZ) \stackrel{j_!}{\surj} H^{g+1} (A ,\bZ) \surj \: (P^{g-1})^* \}.
\]
It follows from the results of \cite{Hazama94} that, for $(A, \T)$ generic (i.e., outside a countable union of Zariski closed subsets of the moduli space $\cA_g$), any rational sub-Hodge structure of $H^{g-1} (A, \bQ)$ of coniveau $1$ or more is contained in $\theta \cup H^{g-3} (A, \bQ)$. Therefore any rational sub-Hodge structure of $H^{g-1} (\T , \bQ)$ of coniveau $1$ or more is contained in $\bH_{\bQ} = \bH\otimes \bQ$.

Note that $\bH_{\bQ} = \bK_{\bQ} \oplus \theta \cup H^{g-3}(A, \bQ)$. Therefore the Hodge conjecture for $\bH_{\bQ}$ is equivalent to the Hodge conjecture for $\bK_{\bQ}$.

We discuss the cases $g=4$ and $5$ in Sections \ref{secA4} and \ref{secA5}. In Sections 2 and 3 below we review two of the main tools used in the proofs for $g=4,5$: Prym varieties and $n$-gonal constructions.

\section{Useful facts about Prym varieties}

Let ${\cR}_{g+1}$ be the coarse moduli space of admissible (in the sense of \cite{beauville771}) double covers of stable curves of genus $g+1$. The moduli space ${\cR}_{g+1}$ is a partial compactification of the moduli space of \'etale double covers of smooth curves. Beauville \cite{beauville771} showed that the Prym map $\cP_g : {\cR}_{g+1} \lra {\cA}_{g}$ is proper. The prym map $\cP_g$
associates to each admissible double cover $(\pi: \tX \lra X)$ of a stable
curve $X$ of genus $g+1$ its Prym variety
\[
\begin{array}{ccc}
P(\tX,X)& := & Im(1-\sigma^{*}:J(\tX) \lra J(\tX))\\
        &  = & Ker^{0}(\nu: J(\tX) \lra J(X))
\end{array}
\]
where $\sigma$ is the involution interchanging the two sheets of $\pi$,
$\nu : Pic(\tX) \lra Pic(X)$ is the norm map and by $Ker^0(\nu)$ we mean the
component of the identity in the kernel of $\nu$.
For general background on the Prym construction we refer to \cite{beauville771}
and \cite{mumford74}.
The Prym maps $\cP_4$ and $\cP_5$ are surjective \cite{beauville771}.\\

There is a useful parametrization of the Prym variety of a covering.
Consider the following subvarieties of $Pic^{2g}(\tX)$
$$A^{+}:= \{D \in Pic^{2g}(\tX): \nu (D) \cong \omega_{X}, h^{0}(D) {\rm \;\;
even}\}$$
$$A^{-}:=\{D \in Pic^{2g}(\tX): \nu (D) \cong \omega_{X},  h^{0}(D) {\rm \;\;
odd}\}$$
Both are principal
homogeneous spaces over $A$. The divisor $\T$ is
a translate of
$$ \T^+ =\{ L \in A^{+}: h^{0}(L) > 0 \}.$$
For each $D \in A^{-}$ we have an embedding (see \cite[p. 97]{I3})
$$\phi_D : \tX \lra A^{+} \subset J(\tX)\quad ; \quad x \mapsto D(L_x - \sigma (L_x))$$
where $L_x$ is an effective Cartier divisor of degree 1 on $\tX$ with
support $x$. The image $\tX_D$ of such a morphism is called a Prym-embedding of
$\tX$ or a Prym-embedded curve.\\
Note that $\tX_D \subset \T^+$ if and only if $h^0 (D) \geq 3$. The set of Prym-embeddings of $\tX$ in $\T^+$ is therefore parametrized by
\[
\lambda(\tX) := \{ D \in A^{-}: h^{0}(D) > 1 \}.
\]
The involution $\sigma : \tX \ra \tX$ induces an involution, also denoted $\sigma$:
\[
\sigma : \lambda(\tX) \lra \lambda(\tX) \quad ; \quad D \longmapsto \sigma^* D.
\]
We put
\[
\lambda (X) := \lambda (\tX) / \sigma.
\]
Note that $\sigma$ has finitely many fixed points in $A^-$, hence at most finitely many fixed points in $\lambda(\tX)$.

\section{The $n$-gonal construction}

Suppose given an \'etale double cover $\kappa : \tX \ra X$ of a smooth curve $X$ of genus $g+1$. Suppose also given a non-constant map $X \ra \bP^1$ of degree $n$.

Sending a point of $\bP^1$ to the sum of the points of $X$ above it, allows us to think of $\bP^1$ as a subscheme of $X^{(n)}$, where $X^{(n)}$ is the $n$-th symmetric power of $X$. Let $\tC\subset\tX^n$ be the curve defined by the
fiber product diagram

\begin{eqnarray} \label{eq1}
\tC & \inj & \tX^{ (n) } \nonumber \\
\downarrow & & \downarrow \kappa^{ (n) } \\
\bP^1 & \inj & X^{ (n) }. \nonumber 
\end{eqnarray}

In other words, the curve $\tC$ parametrizes the $2^n$ points lifting the same point of
$\bP^1$. The involution $\sigma$ also induces an involution on $\tC$, still denoted $\sigma$. The curve $\tC$ has two connected components $\tC_1$ and $\tC_2$ which are exchanged under $\sigma$ if $n$ is odd. If $n$ is even, $\sigma$ leaves each component globally invariant (see e.g. \cite{beauville82}).

\section{The case $g=4$}\label{secA4}

Since $dim({\cA}_{4})=10$ and
$dim({\cR}_{5})=12$, the fiber $\cP_4^{-1}(A)$ for $A$ generic in $\cA_{4}$
is a smooth surface. When $\T$ is smooth, the fiber is always
a surface and the generic elements of any component of the fiber are double
covers of smooth curves (see \cite{I3} pages 111, 119 and 125).

If $A$ is neither decomposable nor the
Jacobian of a hyperelliptic curve, then $\lambda(\tX)$ is a curve and the Prym variety of the double cover $\lambda(\tX) \ra \lambda(X)$ is isomorphic to $(A, \T)$ (see \cite{I3} p. 119). Sending $\lambda: (\tX,X)$ to $(\lambda(\tX),\lambda(X))$ defines an involution $\lambda$ acting on the fibers of the Prym map $\cP_4$.

To $(A,\T)\in\cA_4$ with smooth $\T$, one can associate a smooth cubic threefold $T$ (\cite{I3}, \cite{donagi92}). The quotient of the fiber $\cP_4^{-1}(A)$ by the involution $\lambda$ can be identified with the Fano surface $F$ of lines on $T$.

Let $\cF$ be the scheme parametrizing the family of
Prym--embedded curves {\rm inside} $\T$.

It follows that the fiber of the natural
projection 
$$ \cF \lra \cP_4^{-1}(A)$$ 
over the point $(\tX,X) \in
\cP_4^{-1}(A)$ is the curve $\lambda(\tX)$. In particular, the
dimension of $\cF$ is three. In general, $\cF$ might be
singular, but for $A$ general $\cF$ is smooth, see \cite[Section 3]{I4}.

Let $\cC\ra \cF$ be the tautological family over $\cF$ with the natural map to $\T^+$:
\[
\xymatrix{\cC\ar[r]^-{q}\ar[d]^-{r}&\T^+\\
\cF.}
\]

\begin{theorem}(\cite{I4}) For $(A,\T)$ general in $\cA_4$, the image of the Abel-Jacobi map $q_!r^*: H^5(\cF,\bQ)
\ra H^3(\T^+,\bQ)$ is equal to $\bH_{\bQ}$.
\end{theorem}

It is in fact proved in \cite{I4} that for any $(A,\T)$ with $\T$ smooth, the image of $\cF$ in the intermediate Jacobian of $\T$ generates the abelian subvariety associated to $\bH$.

\section{The case $g=5$}\label{secA5}

The spaces $\cA_5$ and $\cR_6$ both have dimension $15$ and $\cP_5$ is surjective. So $\cP_5$ is generically finite and its degree was computed in \cite{donagismith81} to be $27$.

In \cite{I17}, we use the $5$-gonal construction to construct a family of curves in $\T$ as follows. 

Let $X$ be a smooth curve of genus $6$ with an \'etale double cover
$\tX$ of genus 11. For a pencil $M$ of degree $5$ on $X$ consider the curve $B_M$
defined by the pull-back diagram
\[
\begin{array}{ccc}
B_M& \subset & \tX^{ (5) } \\
\downarrow & &\downarrow \\
\bP^1 = |M| & \subset & X^{ (5) }.
\end{array}
\]
By \cite[p. 360]{beauville82} the curve $B_M$ has two isomorphic connected
components, say $B_{M}^1$ and $B_{M}^2$.
Put $M ' = | K_X - M |$. Then,
for any $D\in B_M\subset\tX^{ (5) }$ and any $D'\in B_{M'}\subset\tX^{ (5)
}$, the push-forward to $X$ of $D+D'$ is a canonical divisor on
$X$. Hence the image of
\[\begin{array}{ccc}
B_M\times B_{M'} & \lra & \Pic^{ 10 }\tX \\
(D,D') &\longmapsto &\cO_{\tX } (D+D' )
\end{array}
\]
is contained in $A^+\cup A^-$. If we have labeled the connected
components of $B_M$ and $B_{M'}$ in such a way that $B_{M}^1\times B_{M'}^1$ maps
into $A^+$, then $B_{M}^2\times B_{M'}^2$ also maps into $A^+$ while $B_{M}^1\times
B_{M'}^2$ and $B_{M}^2\times B_{M'}^1 $ map into $A^-$. By construction, the images of $B_{M}^1\times B_{M'}^1$ and $B_{M}^2\times B_{M'}^2$ lie in $\T^+$.

To obtain a family of curves in $\T^+$, we globalize the above construction.

The scheme $G^1_5 (X)$ parametrizing linear systems of degree
$5$ and dimension at least $1$ on $X$ has a determinantal structure
which is a smooth surface for $X$ sufficiently general (see, e.g., \cite[Chapter V]{ACGH}). 
The universal family $P_5^1$ of divisors of the elements of $G^1_5$ is a
$\bP^1$ bundle over $G^1_5$ with natural maps
\[
\begin{array}{ccc}
P^1_5 & \lra & X^{ (5) } \\
\downarrow & & \\
G^1_5 & & \\

\end{array}
\]
whose pull-back via $\tX\ra X$ gives us the family of the
curves $B_M$ as $M$ varies:
\[
\begin{array}{ccc}
B & \lra & \tX^{ (5) } \\
\downarrow & &\downarrow \\
P^1_5 & \lra & X^{ (5) } \\
\downarrow & & \\
G^1_5 . & & 
\end{array}
\]
The parameter space of the connected components of the curves $B_M$ is an \'etale double cover $\widetilde{G}^1_5$ of
$G^1_5$. If we make a base change,
\[
\xymatrix{B_1\cup B_2\ar[r]\ar[d]&B\ar[d]\\
\widetilde{G}^1_5\ar[r]&{G}^1_5 \ \ ,}
\]
 the family of curves $B$ splits into to components $B_1$ and $B_2$, where $B_1\stackrel{r}\ra\widetilde{G}^1_5$ is the tautological family, i.e, the fiber of $B_1$ over a point $(|M|, B_M^i)\in\widetilde{G}^1_5$ is exactly the curve $B_M^i$.

The family of curves $\cF$ is then defined to be the fiber product
\[
\xymatrix{\cF\ar[r]\ar[d]&B_1\ar[d]^-{\iota\circ r}\\
B_1\ar[r]^-{r}&\widetilde{G}^1_5,}
\]
where $\iota$ is the involution on $\widetilde{G}^1_5$ sending $(|M|,B^i_M)$ to $(|M'|,B^i_{M'})$.

 For $X$ sufficiently general, we obtain a family $\cF$ of smooth curves of genus $25$ over a smooth threefold $B_1$ in the theta divisor $\T^+$ of $A^+ \cong P(\tX, X)$:
 \[
 \xymatrix{\cF\ar[r]^-{q}\ar[d]^{r}&\T^+\\
 B_1.}
 \]
 
The main result of \cite{I17} is 
 
 \begin{theorem} \label{thmI17} For a general Prym variety $P(\tX,X)$, the image of the Abel-Jacobi map $q_!r^*: H^4(B_1,\bQ)\ra H^4(\T^+,\bQ)$ is equal to $\bH_\bQ$.
 \end{theorem}

Note that $H^4(B_1,\bQ)$ is a level $2$ Hodge structure isomorphic to $H^2(B_1,\bQ)$ under the Lefschetz isomorphism. Combining Theorem \ref{thmI17} with the main result of \cite{Hazama94},  we obtain
\begin{corollary} For $(A, \T)$ in the complement of countably many proper Zariski closed subsets of $\cA_5$, the general Hodge conjecture holds for $\T$.
\end{corollary}

As far as we are aware, the primitive cohomology of the theta divisor of an abelian fivefold is the first nontrivial case of a proof of the Hodge conjecture for a family of fourfolds of general type. The proof was considerably more difficult than the case of of the theta divisor of the abelian fourfold worked out in \cite{I4} and required a difficult degeneration to the case of a Jacobian. The computation of the Abel-Jacobi map was broken into computations on different graded pieces of the limit mixed Hodge structures of $\cF$ and $\T$, see \cite{I17} for the full details.

\section{Higher dimensional cases}

As is often the case with deep conjectures such as the Hodge conjecture, the level of difficulty goes up exponentially with the dimension of the varieties concerned or, perhaps more accurately, with their Kodaira dimension.

In higher dimensions a general principally polarized abelian variety is no longer a Prym variety. It is however, a Prym-Tyurin variety \cite{welters871}. This is, up to now, the most promising generalization of Prym varieties. For any curve $C$ generating the abelian variety $A$ as a group, pull-back on the first cohomology induces a map $A \ra JC$ which has finite kernel. Assume given a curve $C$ and a symmetric correspondence $D\subset C\times C$. Denote $\T_C$ a Riemann theta divisor on $JC$, i.e., a translate of the variety $W_{g-1} \subset Pic^{g-1} C$ of effective divisor classes. Also denote by $D$ the endomorphism $JC \ra JC$ induced by $D$. We have the following

\begin{definition}
We say that $(A, \T)$ is a Prym-Tyurin variety for $(C, D)$ if there exists a positive integer $m$ such that $D$ satisfies the equation
\[
(D-1)(D+m-1) = 0
\]
and there is an isomorphism $A\cong \im (D-1)$ inducing an algebraic equivalence $\T \equiv m \T_C |_{\im (D-1)}$.

The integer $m$ is called the {\em index} of the Prym-Tyurin variety.
\end{definition}

To find a family of curves in $\T$ that would give an answer to the Hodge conjecture for $\bH_{\bQ}$ or $\bK_{\bQ}$ (as in the cases $g=4, 5$), we need an explicit Prym-Tyurin structure on $(A, \T)$. In particular, we need to know at least one value of the index $m$. In general, there is very little known about the indices of ppav. In dimension $6$ however, we have the following (see \cite{izadiA6}).

\begin{theorem}
For $(A,\T)$ general of dimension $6$, there is a Prym-Tyurin structure $(C,D)$ of index $6$ on $(A,\T)$.

Furthermore, there is a morphism $\pi : C\ra \bP^1$ of degree $27$ such that the Galois group of the associated Galois cover $X \ra \bP^1$ is the Weyl group $W(E_6)$. The morphism $\pi$ has $24$ branch points and above each branch point there are $6$ simple ramification points in $C$. If $P\in \bP^1$ is not a branch point of $\pi$, the action of $W(E_6)$ on $\pi^{-1} (P)$ gives an identification of $\pi^{-1} (P)$ with the set of lines on a smooth cubic surface such that the restriction of the correspondence $D$ to $\pi^{-1} (P) \times \pi^{-1} (P)$ can be identified with the incidence correspondence of lines on a smooth cubic surface.
\end{theorem}

Prym-Tyurin structures for correspondences obtained from covers with monodromy group Weyl groups of Lie algebras were constructed by Kanev \cite{kanev95} (also see \cite{langerojas08}) who also proved irreducibility results for some of the Hurwitz schemes parametrizing such covers \cite{kanev06}. In particular, Kanev proved that the Hurwitz scheme parametrizing covers as in the above theorem is irreducible.

\section{Open problems}

\begin{description}
\item[Irreducibility]
It would be interesting to know whether the Hodge structure $\bK$ is irreducible. This is trivially true in dimensions up to $3$ and follows from the results of \cite{I4} in dimension $4$. In dimension $5$ this would simplify the computation of the Abel-Jacobi map hence shorten the proof of \cite{I17}.

\item[$E_6$ structure when $g=5$]
The monodromy group of the Prym map $\cR_6 \ra \cA_5$ is the Weyl group $W(E_6)$ of the exceptional Lie algebra $E_6$ (see \cite[Theorem 4.2]{donagi92}). Also, the lattice $\bK$ has rank $78$ for $g=5$ which is equal to the dimension of $E_6$. So one might wonder whether it is possible to define a natural isomorphism between $\bK_{\bC} := \bK\otimes \bC$ and $E_6$.

\item[Generalization of the $n$-gonal construction] As we saw the $5$-gonal (or pentagonal) construction is used in the construction of the family of curves in dimension $5$ and the $4$-gonal (or tetragonal) construction is important for understanding the family of curves in dimension $4$. Therefore, one can ask whether there is a good generalization of the $n$-gonal construction for correspondences (in analogy with double covers) that would allow one to construct a good family of curves in higher dimensions.

\item[Catalan numbers]
The $g$-th Catalan number can be directly defined as
\[
C_g := \frac{1}{g+1}{2g \choose g}
\]
and is the solution to many different counting problems (see, e.g., \cite{koshy08}). For instance, $C_g$ is the number of permutations of $g$ letters that avoid the pattern $1,2,3$. This means that, if we represent a permutation $\sigma$ by the sequence $s(\sigma) := (\sigma(1), \sigma (2) , \ldots, \sigma(g))$, then the sequence $s(\sigma)$ does not contain any strictly increasing subsequence of length $3$. Or, $g! - C_g = \dim (\bK_{\bQ})$ is the number of permutations of $g$ letters that contain the pattern $1,2,3$ (i.e., $s(\sigma)$ does contain a strictly increasing subsequence of length $3$). An interesting question would be to find degenerations of $\T$, i.e., $\bK$, that illustrate some of these counting problems. For instance, a degeneration of $\T$ and $\bK$ that would exhibit a natural basis of $\bK$ indexed by the permutations of $g$ letters that contain the pattern $1,2,3$.

\end{description}

%\bibliographystyle{alpha}
%\bibliography{biblio}

\newcommand{\etalchar}[1]{$^{#1}$}

\end{document}